\documentclass[12pt]{article}
\usepackage{geometry}
\usepackage[T1]{fontenc}
\usepackage{amsmath}
\usepackage{amsfonts}
\usepackage{amssymb}
\usepackage{mathtools}
\usepackage{calrsfs}
\usepackage[shortlabels]{enumitem}
\usepackage{amsthm}
\usepackage{tikz}
\usepackage{tikz-cd}
\usepackage{hyperref}
\usepackage{quiver}
\usepackage{comment}
\usepackage{stmaryrd}
\usepackage[natbib, style=alphabetic, doi=false,isbn=false,url=false,eprint=false]{biblatex}

\usetikzlibrary{decorations.pathreplacing}

\usepackage{ytableau}

\author{MATEUSZ LOWIEL}
\title{\textbf{QUIVER GRASSMANNIANS ASSOCIATED TO NILPOTENT CYCLIC REPRESENTATIONS DEFINED BY SINGLE MATRIX}}
\date{}
\DeclareMathAlphabet{\pazocal}{OMS}{zplm}{m}{n}

\newcommand{\proj}[2]{\mathbb{P}_{#1}^{#2}}

\newcommand{\qgr}[2]{\text{Gr}_{\mathbf{#1}}(#2)}
\newcommand{\clBB}[1]{\overline{\mathcal{C}_{#1}}}

\newtheorem{manualtheoreminner}{Theorem}
\newenvironment{manualtheorem}[1]{%
	\renewcommand\themanualtheoreminner{#1}%
	\manualtheoreminner
}{\endmanualtheoreminner}

\newtheorem{manualpropinner}{Proposition}
\newenvironment{manualprop}[1]{%
	\renewcommand\themanualpropinner{#1}%
	\manualpropinner
}{\endmanualpropinner}

\theoremstyle{plain}
\newtheorem{notation}[subsection]{Notation}
\newtheorem{defi}[subsection]{Definition}
\newtheorem{stw}[subsection]{Proposition}
\newtheorem{tw}[subsection]{Theorem}
\newtheorem{obs}[subsection]{Observation}

\theoremstyle{remark}
\newtheorem*{uw}{Remark}
\newtheorem{przyk}[subsection]{Example}

\begin{document}
	\maketitle
	\begin{abstract}
		\noindent In the present paper we study the geometry of the closed Białynicki-Birula cells of the quiver Grassmannians associated to a nilpotent representation of a cyclic quiver defined by a single matrix. For the special case, where we choose subrepresentations of dimension $\mathbf{1}=(1,\dots,1)$, the main result of this paper is that the closed Białynicki-Birula cells are smooth. We also discuss the multiplicative structure of the cohomology ring of such spaces. Namely, we describe the so-called Knutson-Tao basis in context of the basis of equivariant cohomology that is dual to fundamental classes in equivariant homology.
	\end{abstract}
	\tableofcontents
	\section*{Introduction}
	The theory of quiver Grassmannians, first introduced by Schofield in \cite{GeneralRepresentationsOfQuivers} has its roots in representation theory, but it found its way to other fields, such as algebraic geometry. This theory, from the perspective of algebraic geometry, is a rich one -- we know that these spaces describe all projective varieties as was shown in \cite{EveryProjectiveVarietyIsQuiverGrassmannian}, so they are an attractive object to study. For acyclic quivers many properties such as smoothness, irreducibility or dimension have known criteria, for more information one can consult \cite{3LecturesOnQuiverGrassmannians}. 
	
	More recently, M. Lanini and A. P{\"u}tz in \cite{GKMTheoryForCyclicQuiverGrassmannians} define torus actions on the quiver Grassmannians associated to the so-called nilpotent representations of cyclic quivers, which endows the quiver Grassmannians with a GKM-variety structure and so lets us compute equivariant cohomology of these spaces. These spaces were further analyzed in  \cite{PermutationActionsOnQuiverGrassmannians}, \cite{DesingularizationsOfCyclicQuiverGrassmannians} in context of group actions and resolution of singularities respectively.
	
	This paper delves into the GKM-variety structure and the Białynicki-Birula decomposition of these quiver Grassmannians in a special case. Let $\Delta_n$ be the equioriented cycle with $n$ vertices and $\mathcal{M}$ be a $\Delta_n$-representation that has the same matrix $J:\mathbb{C}^N\to\mathbb{C}^N$ over each arrow in $\Delta_n$. In this paper we study the quiver Grassmannian $X=\qgr{1}{\mathcal{M}}\subseteq\prod_{i=1}^n\proj{}{N-1}$ of subrepresentations of dimension vector $\mathbf{1}\in\mathbb{Z}^{n}$. Without loss of generality $J$ can be assumed to be a Jordan matrix with $M$ Jordan blocks. Following \cite{GKMTheoryForCyclicQuiverGrassmannians} we introduce a torus action on $X$ and study its $T$-equivariant cohomology. Authors of \cite{TotallyNonnegativeGrassmanniansGrassmannNecklacesAndQuiverGrassmannians} tackled a similar problem -- authors considered the quiver Grassmannian $\qgr{k}{\mathcal{M}}$, where $\mathbf{k}=(k,\dots,k)$, but $\mathcal{M}$ was defined by $J$ which is a single Jordan block of size $n$. In that paper one can find, for example, the Poincar{\'e} polynomials of such quiver Grassmannians, resolution of singularities and the connection between these quiver Grassmannians and totally nonnegative Grassmannians. 
	
	The Białynicki-Birula cells form a stratification of the quiver Grassmannian into affine cells, as was shown in \cite{GKMTheoryForCyclicQuiverGrassmannians}. In our case  we prove the following two results concerning the structure of the closed Białynicki-Birula cells
	\begin{manualtheorem}{$\ref{ClBBCellsSmooth}$}
		Let $p\in X^{T}$ be a fixed point. Then the closure of the BB-cell $\mathcal{C}_p$ is a product of smooth quiver Grassmannians.
	\end{manualtheorem}
	The nilpotent Jordan matrix $J$ corresponds to a partition of $N$ and so we can visualize coordinates of $\mathbb{C}^N$ using the Young tableau $\mathcal{J}$ of shape associated to $J$. Appropriately identifying the basis vectors of $\mathbb{C}^N$ with boxes of $\mathcal{J}$ leads to the definition of twisted lexicographic order on this basis. More precisely, suppose that $J$ has $M$ blocks and the $l$-th block is of size $j_l$. Let $v_i^l$ be the $i$-th Jordan basis vector of the $l$-th block of $J$. We then set $v_i^l<v_k^s$ if either $j_s-k<j_l-i$ or $j_s-k=j_l-i$ and $l<k$. The closed BB-cells form a poset with respect to inclusion and we prove that these two relations are inverse to each other
	\begin{manualtheorem}{$\ref{OrderOfBBCells}$}
		For two fixed points $p,q\in X^T$ we have $\clBB{p}\subseteq\clBB{q}$ if and only if $q_i\le p_i$ for every $i\in\mathbb{Z}/n\mathbb{Z}$.
	\end{manualtheorem}
	We can use the visual nature of coordinates to visualize possible edges in the GKM-graph of $X$. Namely, each fixed point $p\in X^T$ consists of choosing the basis vectors $v_i^l$ over each vertex. Thus, we can divide $p$ into movable parts, i.e. segments of these lines in $\mathbb{C}^N$ that do not go to $0$ under $J$. Each movable part can be seen as a coloring of first couple of boxes in a fixed row in $\mathcal{J}$ and a mutation of movable part is defined to be either a coloring of the same amount of boxes, but in a lower row or two colorings -- we split the first coloring into two. We then show that these mutations define edges in the GKM-graph of $X$ in the following Proposition
	\begin{manualprop}{$\ref{GKM}$}
		Let $p,q\in X^T$ be two fixed points. Then an arrow $p\to q$ in the GKM graph of $X$ exists if and only if $q$ is obtained from $p$ by a mutation of exactly one movable part in $p$.
	\end{manualprop}
	We also give a way to get the label of such an edge in Proposition $\ref{GKMedges}$.
	
	The closed Białynicki-Birula cells from a basis of equivariant homology and since they are smooth, integration along these cells can be done using the Atiyah-Bott, Berline-Vergne, which heavily simplifies actual computations. We can construct the dual basis $\{p^x\in H_T^*(X):x\in X^T\}$ and this basis is immediately seen to be a Knutson-Tao basis of $H_T^*(X)$, which were defined in \cite{PermutationRepresentationsOnSchubertVarieties}. Again, as the cells are smooth, then computing things such as structure constants $c_{x,y}^z$, defined by
	\begin{equation*}
		p^x\cdot p^y = \sum_{z\in X^T}c_{x,y}^zp^z
	\end{equation*}
	is immediate, since this comes down to linear algebra (albeit linear algebra over large vector spaces).
	
	This paper is constructed as follows -- Chapter $\ref{BasicDefinitions}$ is meant to fix notation and recall some very basic notions regarding the quiver Grasssmannians and nilpotent representations. Chapter $\ref{TorusActions}$ recalls the construction of actions of two tori on quiver Grassmannians for nilpotent $\Delta_n$-representations from \cite{GKMTheoryForCyclicQuiverGrassmannians}, but only for our case. Chapter $\ref{BBDecomposition}$ is about the Białynicki-Birula decomposition of this particular family of quiver Grassmannians. Namely, we recall a result from \cite{GKMTheoryForCyclicQuiverGrassmannians}, that the induced Białynicki-Birula decomposition is a cellular decomposition into affine cells. We further investigate how these cells look like after taking the closure in the ambient space of the quiver Grassmannian, which is a product of projective spaces. This chapter proves that the closed Białynicki-Birula cells are smooth. In Chapter $\ref{GKMStructure}$ we describe the GKM graph of these quiver Grassmannians, which was also computed in \cite{GKMTheoryForCyclicQuiverGrassmannians}, but we provide a description using colorings of Young tableaux. In Chapter $\ref{MultiplicativeStructure}$ we define the basis dual to the equivariant fundamental classes in equivariant homology and discuss computing structure constants of this basis.
	\section{Basic definitions and notation}\label{BasicDefinitions}
	Let $T=(\mathbb{C}^\times)^r$ be an algebraic torus. For a $T$-space $X$ we define its equivariant cohomology via the Borel construction. We denote by $H_T^*(X)$ the equivariant cohomology ring with coefficients in $\mathbb{Q}$ (for a proper introduction into equivariant cohomology see \cite{EquivariantCohomologyInAlgebraicGeometry}). 
	
	Now suppose that $X$ is a $\mathbb{C}^\times$-space and we decompose the fixed point set into connected components as follows
	\begin{equation*}
		X^{\mathbb{C}^\times} = \bigcup_{k\ge0} X_k,
	\end{equation*}
	then we define the \textbf{Białynicki-Birula cells} or the \textbf{BB-cells} as
	\begin{equation*}
		\mathcal{C}_k = \{x\in X:\lim_{\lambda\to0}\lambda\cdot x\in X_k\}.
	\end{equation*}
	Together these cells form the \textbf{BB-decomposition} of $X$. If these cells are affine spaces, then $X$ is an equivariantly formal $\mathbb{C}^\times$-variety.
	
	Let $Q=(Q_0,Q_1)$ be a finite and connected quiver and $M=(M_i,M_{\alpha})_{i\in Q_0,\alpha\in Q_1}$ be a representation of $Q$ over $\mathbb{C}$ (equivalently, a $\mathbb{C}Q$-module). We denote by ${\dim M\in\mathbb{Z}_{\ge0}^{|Q_0|}}$ the dimension vector of $M$, i.e. $(\dim M)_{i}=\dim_{\mathbb{C}} M_i$ (for a proper introduction to representations of quivers one can consult \cite{ElementsOfTheRepresentationTheoryOfAssociativeAlgebras}). For a fixed vector $\mathbf{e}\le \dim M$ we denote by $\qgr{e}{M}$ the quiver Grassmannian of subrepresentations of $M$ of dimension $\mathbf{e}$. These spaces are known to have a (projective) scheme structure given by their embedding into a product of usual Grassmannians, since a subrepresentation of $M$ is determined by a choice of subspaces for all $M_i$'s. Note that these spaces in general are not varieties as they need not be irreducible or even reduced. \\
	
	Let $M\in\text{rep}(Q)$ be a representation and $B$ be the set of basis elements of each vector space in $M$. We divide $B$ into subsets corresponding to each vertex -- let $B^i\subset B$ be the subset of $B$ that is the basis of $M_i$. \textbf{The coefficient quiver of $M$} with respect to the basis $B$ is defined as the following quiver: 
	\begin{enumerate}
		\item The set of vertices is $B$.
		\item For $v_k^i\in B^i, v_l^j\in B^j$ an arrow $v_k^i\to v_l^j$ exists if and only if there exists an arrow $\alpha:i\to j$ such that the $l$-th coordinate of $M_\alpha(v_k^i)$ in $B^j$ is non-zero.
	\end{enumerate}
	We denote this coefficient quiver by $Q(M,B)$. This tool was used to find a well-suited tori actions on quiver Grassmannians, for example in \cite{GKMTheoryForCyclicQuiverGrassmannians} and this paper is following this construction. \\
	
	In this paper we will be interested in a special family of representations of a certain family of quivers. Namely, let $\Delta_n$ be the oriented cycle quiver, that is an equioriented $n$-gon. We label both the vertices and edges using the set $\mathbb{Z}/n\mathbb{Z}$. 
	
	A $\Delta_n$-representation $M$ is called \textbf{nilpotent} if there exists a number $N>0$ such that $M$, considered as a $\mathbb{C}\Delta_n$-module, is annihilated by the ideal of paths of length at least $N$. Explicitly, this means that for any $i\in Q_0$ the composition of $N$ consecutive maps from $M$ vanishes.
	
	\begin{przyk}\label{MainObject}
		A basic example of a nilpotent representation is a representation that has the same vector space, say of dimension $N$, over each vertex and the same nilpotent map over each arrow. Since we can simultaneously change basis, we can without loss of generality assume that over each vertex we have $\mathbb{C}^N$ and our map is a Jordan matrix $J:\mathbb{C}^N\to\mathbb{C}^N$, whose Jordan-block decomposition is given by
		\begin{equation*}
			J = J_1\oplus \dots\oplus J_{M}, \text{ where } J_i\in M_{j_i\times j_i}(\{0,1\}), j_1\ge j_2\ge\dots\ge j_M.
		\end{equation*}
		As $J$ must be nilpotent, each Jordan block has zeroes on the diagonal. We will denote this representation by $\mathcal{M}=(\mathcal{M}_i=\mathbb{C}^N,\mathcal{M}_\alpha=J)_{i,\alpha\in\mathbb{Z}/n\mathbb{Z}}$.
	\end{przyk}
	Following \cite{GKMTheoryForCyclicQuiverGrassmannians} we can construct a GKM-variety structure on any quiver Grassmannian associated to any nilpotent $\Delta_n$-representation and we will use this result to study the geometry of $\qgr{1}{\mathcal{M}}$, where $\mathbf{1}=(1,\dots,1)$. \\
	
	Throughout this paper we will talk about subspaces defined as a linear span of some basis vector and we will shorthand $\text{span}(v_1,\dots,v_n)$ by $\langle v_1,\dots,v_n\rangle$. We denote the set $\{1,\dots,M\}$ by $[M]$. 
	\section{Torus actions on the main object}\label{TorusActions}
	In this section we recall known results from \cite{GKMTheoryForCyclicQuiverGrassmannians} applied to our example of $X:=\qgr{1}{\mathcal{M}}$, where $\mathcal{M}$ is the representaion defined in Example ${MainObject}$ and $\mathbf{1}=(1,\dots,1)$. This space clearly embeds into $\mathbb{P}:=(\mathbb{P}^{N-1})^n$ and authors of \cite{GKMTheoryForCyclicQuiverGrassmannians} in fact define actions of two tori on $\mathbb{P}$ that fixes $X$. 
	
	Let us start with the $\mathbb{C}^\times$-action. Fix $l\in\{1,\dots,M\}$ and let $V_l\subseteq\mathbb{C}^N$ be the subspace fixed by the $l$-th Jordan block in $J$. We denote the Jordan basis of $V_l$ by $v_1^l,\dots,v_{j_l}^l$. Consider a linear action of $\mathbb{C}^\times$ on $\mathbb{C}^N$, defined on basis vectors $v_i^l\in\mathbb{C}^N$ as follows
	\begin{equation}\label{C-action}
		\lambda\cdot v_i^l := \lambda^{l-M(j_l-i)}v_i^l, \text{ for }\lambda\in\mathbb{C}^\times.
	\end{equation}
	In other words, the weight of $v_{j_l}^l$ is $l$ and as we decrease the index of the basis vector we decrease the weight by $M$ (so long as the vectors lie in the same $V_l$). Note that we act on each coordinate by different weights. After we pass to the projective space $\proj{}{N-1}$ the fixed point set under the induced torus action is the set of (linear spans of) standard basis vectors. We extend this action to the diagonal action on $\mathbb{P}$. It requires a check, that $X$ is $\mathbb{C}^\times$-invariant under this action and we refer the reader to \cite[Lemma 1.1]{QuiverGrassmanniansAssociatedWithStringModules} for the proof.
	
	Next, we describe an action of $T=(\mathbb{C}^\times)^{nM+1}$ on $\mathbb{P}$ that fixes $X$, but omit describing the whole construction here and state the resulting torus action; readers interested in where this action comes from are referred to \cite[Subsection 5.2]{GKMTheoryForCyclicQuiverGrassmannians}. We start with the $0$-th vertex and define the following torus action on basis vectors
	\begin{equation*}\label{T-action}
		(t_0,\{t_{r,s}\}_{r\in\mathbb{Z}/n\mathbb{Z},s\in[M]})\cdot v_{i}^l := t_0^{i-1}t_{j_l-i,l}v_i^l, \text{ for }(t_0,\{t_{r,s}\}_{r\in\mathbb{Z}/n\mathbb{Z},s\in[M]})\in T.
	\end{equation*}
	This naturally extends to $\proj{}{N-1}$ and to extend this to $\mathbb{P}$ we use an automorphism $\rho:T\to T$, which we define as follows
	\begin{align}\label{Rotation}
		\rho(t_0):=t_0, && \rho(t_{j,k}) := t_{j+1,k}.
	\end{align}
	Now, if $v_i^l$ lies over the $j$-th vertex, then we set
	\begin{equation*}
		(t_0,\{t_{r,s}\}_{r\in\mathbb{Z}/n\mathbb{Z},s\in[M]})\cdot v_i^l := \rho^j(t_0,\{t_{r,s}\}_{r\in\mathbb{Z}/n\mathbb{Z},s\in[M]})\cdot v_i^l=t_0^{i-1}t_{j_l-i+j,l}v_i^l.
	\end{equation*}
	This torus action makes $X$ a $T$-equivariant subvariety of its ambient space, by \cite[Lemma 5.10]{GKMTheoryForCyclicQuiverGrassmannians}. 
	
	\begin{przyk}
		Consider the case of a Jordan matrix with two blocks $J=J_3\oplus J_2$ and $3$ vertices. Then the $\mathbb{C}^\times$-action over each vertex is given by
		\begin{equation*}
			\lambda\cdot (a_1^1,a_2^1,a_3^1,a_1^2,a_2^2) = (\lambda^{-2} a_1^1,\lambda^{-1} a_2^1,\lambda^1 a_3^1, a_1^2,\lambda^2 a_2^2)
		\end{equation*}
		and the $T=(\mathbb{C}^\times)^{7}$-action over the zeroth, first and second vertex is given by
		\begin{align*}
			& t\cdot (a_1^1,a_2^1,a_3^1,a_1^2,a_2^2) = (t_{2,1}a_1^1,t_0t_{1,1}a_2^1,t_0^2t_{0,1}a_3^1,t_{1,2}a_1^2,t_{0,2}a_2^2),\\
			& t\cdot(a_1^1,a_2^1,a_3^1,a_1^2,a_2^2) = (t_{0,1}a_1^1,t_0t_{2,1}a_2^1,t_0^2t_{1,1}a_3^1,t_{2,2}a_1^2,t_{1,2}a_2^2),\\
			& t\cdot(a_1^1,a_2^1,a_3^1,a_1^2,a_2^2) = (t_{1,1}a_1^1,t_0t_{0,1}a_2^1,t_0^2t_{2,1}a_3^1,t_{0,2}a_1^2,t_{2,2}a_2^2).
		\end{align*}
	\end{przyk}
	
	There is a beautiful interplay between these two actions -- the $\mathbb{C}^\times$-action induces a BB-decomposition of $X$ that turns out to be a cellular decomposition by \cite[Theorem 5.6]{GKMTheoryForCyclicQuiverGrassmannians}, which implies that $X$ is equivariantly formal. The fixed point sets $X^{\mathbb{C}^\times}$ and $X^T$ are equal and the $\mathbb{C}^\times$-action can be recovered from the $T$-action via a cocharacter $\chi:\mathbb{C}^\times\to T$ defined by
	\begin{equation*}
		\chi(t) := (t^M,\{t^{k-Mj_k+M}\}_{j\in\mathbb{Z}/n\mathbb{Z},k\in[M]}).
	\end{equation*}
	The $T$-variety $X$ is a GKM-variety by \cite[Theorem 6.5]{GKMTheoryForCyclicQuiverGrassmannians} and its GKM graph can be described using the coefficient quiver of $\mathcal{M}$, but in our specific case we can give descriptions of the BB-cells and the GKM of $X$ in a simpler language.
	\begin{notation}\label{CoefficientQuiver}
		After fixing the basis of the vector space $\mathcal{M}$ to be the sets $\{v_i^l\}$ over every vertex, the coefficient quiver of $\mathcal{M}$ with respect to this basis can be drawn as $M$ spirals starting at each vertex of $\Delta_n$ and the $l$-th spiral spans $j_l$ vertices. We adopt the notation from \cite{GKMTheoryForCyclicQuiverGrassmannians} and denote this quiver by $Q(\mathcal{M})$. 
	\end{notation}
	\begin{przyk}
		Take for example $n=3$ and $J=J_3$ to be a single Jordan block of size $3\times 3$. Then the coefficient quiver $Q(\mathcal{M})$ looks as follows
		\[\resizebox{0.7\textwidth}{!}
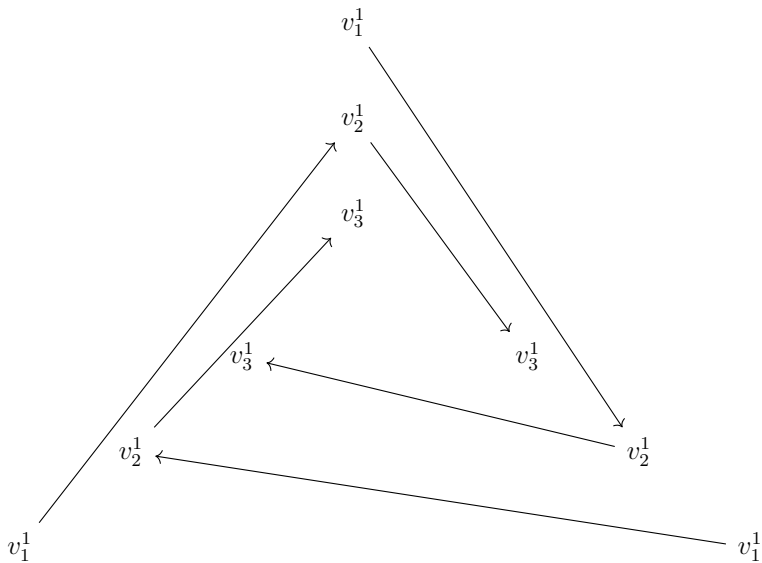
{
			\begin{tikzcd}[ampersand replacement=\&]
			\&\&\& v_1^1 \\
			\&\&\& v_2^1 \\
			\&\&\& v_3^1 \\
			\\
			\&\& v_3^1 \&\&\& v_3^1 \\
			\& v_2^1 \&\&\&\&\& v_2^1 \\
			v_1^1 \&\&\&\&\&\&\& v_1^1
			\arrow[from=1-4, to=6-7]
			\arrow[from=6-7, to=5-3]
			\arrow[from=7-1, to=2-4]
			\arrow[from=2-4, to=5-6]
			\arrow[from=7-8, to=6-2]
			\arrow[from=6-2, to=3-4]
		\end{tikzcd}
	}\]
	\end{przyk}
	Authors of \cite{GKMTheoryForCyclicQuiverGrassmannians} describe the GKM graph of $X$ in \cite[Theorem 6.13]{GKMTheoryForCyclicQuiverGrassmannians}. In later chapters we will remark how the description of the GKM graph of $X$ comes from the one given in \cite{GKMTheoryForCyclicQuiverGrassmannians}.
	\section{Białynicki-Birula decomposition}\label{BBDecomposition}
	Here we will focus on the $\mathbb{C}^\times$-action on $X$, given by Formula $\eqref{C-action}$ and the induced BB-decomposition of $X$. We start with the known description of the fixed point set.
	\begin{stw}[{\cite[Theorem 1, Proposition 1]{QuiverGrassmanniansAssociatedWithStringModules}}]\label{FixedPoints}	
	Fixed points of $X$ under both tori actions consists of representations that over all vertices have vector spaces spanned by the fixed basis $\{v_i^l:i=1,\dots,j_l,l=1,\dots,M\}$. These subrepresentations are in bijective correspondence with successor closed subquivers $S\subseteq Q(\mathcal{M})$, where $Q(\mathcal{M})$ is the coefficient quiver of $\mathcal{M}$ defined in Notation $\ref{CoefficientQuiver}$, such that $S$, restricted to vertices corresponding to a fixed vertex, has exactly one vertex.
	\end{stw}
	In particular any fixed point $p\in X^T$, considered as a subquiver of $Q(\mathcal{M})$ decomposes into connected components. 
	\begin{defi}\label{MovablePart}
		Let $p\in X^T$ be a fixed point and $S_p\subseteq Q(\mathcal{M})$ be the subquiver defined by $p$. Then $S_p$ decomposes into connected succesor-closed subquivers
		\begin{equation*}
			S_p=(\sigma_1,\sigma_2,\dots,\sigma_k),
		\end{equation*}
		where $\sigma_i\subseteq S_p$ are the maximal connected subquivers of $S_p$. Each $\sigma_i$ is called a \textbf{movable part in $p$} and any such $\sigma_i$ defines a representation of the $A$-type equioriented quiver representation
		\begin{equation*}
			\sigma_i \rightsquigarrow (\langle v_{j_l-m}\rangle,\langle v_{j_l-m+1}\rangle,\dots,\langle v_{j_l}\rangle)
		\end{equation*}
		for some $m\ge 0$ and $l=1,\dots,M$. We shall say that $\sigma_i$ is thus of length $m$ and comes from the $l$-th block. To fully define $\sigma_i$ we also need either a starting vertex or the terminal vertex, since $\sigma_i$ corresponds to a part of $p$.
	\end{defi}
	Here we turn to Young tableaux. Our fixed Jordan matrix $J$ corresponds to a Young tableau of shape $J$ and we shall denote this Young tableau by $\mathcal{J}$. Each box thus corresponds to a basis vector $v_i^l$ in $\mathbb{C}^N$. We can fill $\mathcal{J}$ with $\mathbb{C}^\times$-weights that act on these vectors, getting a nice table of weights.
	\begin{przyk}
		Consider the case $J=J_4\oplus J_3\oplus J_3\oplus J_2$. Then we get the following Young tableaux
		\ytableausetup{centertableaux, boxsize=2em}
		\begin{align*}
			J \rightsquigarrow \begin{ytableau}
				v_4^1 & v_3^1 & v_2^1 & v_1^1  \\
				v_3^2 & v_2^2 & v_1^2  \\
				v_3^3 & v_2^3 & v_1^3 \\
				v_2^4 & v_1^4  
			\end{ytableau}  && \mathbb{C}^\times\text{-weights} = \begin{ytableau}
				1 & -3 & -7 & -11  \\
				2 & -2 & -6  \\
				3 & -1 & -5 \\
				4 & 0  
			\end{ytableau} 
		\end{align*}
	\end{przyk}
	We can endow the set of basis elements with a "twisted" lexicographic order. We state the exact definition below.
	\begin{defi}\label{PartialOrder}
		On the set of basis elements $\{v_i^l:i=1,\dots,j_l,l=1,\dots,M\}$ consider the twisted lexicographic order
		\begin{equation*}
			v_i^l < v_j^s \iff \text{either } j_s-j<j_l-i \text{ or } j_s-j=j_l-i \text{ and }l<s.
		\end{equation*}
		This order is of course equivalent to the order of $\mathbb{C}^\times$-weights, i.e.
		\begin{equation*}
			\text{weight}(v_i^l)<\text{weight}(v_j^s) \iff v_i^l<v_j^s.
		\end{equation*}
		We also extend this order to a partial order on fixed points: For two fixed points $p,q\in X^T$ we say that $p<q$ if $p_i\le q_i$ for all $i\in \mathbb{Z}/n\mathbb{Z}$ and $p_j< q_j$ for at least one $j\in\mathbb{Z}/n\mathbb{Z}$.
	\end{defi}
	\begin{uw}
		Note that we can check the inequality $p<q$ for two fixed points $p,q\in X^T$ only at vertices either ending or starting movable parts in $p$. This follows simply from the fact these vertices uniquely determine fixed points.
	\end{uw}
	With this we can parametrize the BB-cells.
	\begin{stw}\label{BBCellsAffine}
		Let $p\in X^{T}$ be a fixed point. Then the BB-cell $\mathcal{C}_p\subset X$ is an affine space.
	\end{stw}
	This was already proven in \cite[Theorem 5.7]{GKMTheoryForCyclicQuiverGrassmannians} in a much more general setting,  but for our special case we will present a proof which will result in an explicit description of these cells.
	\begin{proof}[Proof of Proposition $\ref{BBCellsAffine}$]
		Fix $p\in X^T$. First let us consider $i\in\mathbb{Z}/n\mathbb{Z}$ and we shall study 
		\begin{equation*}
			(\mathcal{C}_p)_i =\{\nu_i\in\proj{}{N-1}:\nu\in\mathcal{C}_p\}.
		\end{equation*}
		Suppose that $p_i=\langle v_j^l\rangle$. Then $(\mathcal{C}_p)_i$ lies in the BB-cell of the ambient variety $\mathcal{C}_i\subseteq\proj{}{N-1}$ associated to $p_i\in\proj{}{N-1}$. This space is easy to describe
		\begin{equation*}
			\mathcal{C}_i = \{u\in\mathbb{C}^N:u_j^l\neq0, u_k^s=0\text{ for all }(k,s)\text{ such that }v_k^s< v_j^l\}.
		\end{equation*}
		Visually, the points of $\mathcal{C}_i$ look as follows
		\ytableausetup{mathmode,boxsize=2.5em}
		\begin{equation*}
			\begin{ytableau}
				u_{j_1}^1 & u_{j_1-1}^1 & \none[\dots] & u_{j+1}^1 & 0 & 0 & \none[\dots] \\
				u_{j_2}^2 & u_{j_1-1}^1  & \none[\dots] & u_{j+1}^2 & 0 & 0 & \none[\dots] \\
				\none[\vdots] & \none[\vdots] & \none & \none[\vdots] & \none[\vdots] & \none[\vdots] \\
				u_{j_l}^l & u_{j_l-1}^l &  \none[\dots] & u_{j+1}^l & 1 & 0 & \none[\dots]\\ 
				\none[\vdots] & \none[\vdots] & \none & \none[\vdots] & \none[\vdots] & \none[\vdots] \\
				u_{j_M}^M & u_{j_M-1}^M & \none[\dots] & u_{j+1}^M & u_{j}^M & 0 & \none[\dots] \\
			\end{ytableau}
		\end{equation*}
		We can choose $u_j^l=1$ to get a parametrization of $\mathcal{C}_i$ as an affine space. The BB-cell of the quiver Grassmannian is the intersection 
		\begin{equation*}
			\mathcal{C}_p = \left( \prod_{i=0}^{n-1}\mathcal{C}_{i} \right)\cap X.
		\end{equation*}
		To pass to another vertex we distinguish two cases
		\begin{enumerate}
			\item $Jv_i^l=0$. Then for any $u\in\mathcal{C}_i$ we have $Ju=0$ as $v_k^s\ge v_j^l$ implies $Jv_k^s=0$. Thus vectors from $(\mathcal{C}_p)_i$ do not impose any conditions on $(\mathcal{C}_p)_{i+1}$.
			\item $Jv_{i}^l=v_{i+1}^l$ is non-zero. Then for any $\nu\in \mathcal{C}_p$ we have
			\begin{equation*}
				J\nu_i=\nu_{i+1}.
			\end{equation*}
			This simply follows from the fact that $\nu_i$ must be spanned by a vector, which has a non-zero coordinate next to $v_i^l$, which does not go to $0$ under $J$. So vectors in $(\mathcal{C}_p)_{i+1}$ are completely determined by vectors in $(\mathcal{C}_p)_i$.
		\end{enumerate}
		Therefore if we have a movable part in $p$ running through vertices $i$ through $m$, then $(\mathcal{C}_p)_i=\mathcal{C}_i$ is an affine space and $(\mathcal{C}_p)_{i+k}=J^k(\mathcal{C}_i)$ for all $k=0,\dots,m$. Dimension of $\mathcal{C}_i$ is given by
		\begin{equation*}
			\dim(\mathcal{C}_i) = |\{ v_j^s:v_j^s>p_i \}|.
		\end{equation*}
		The whole cell $\mathcal{C}_p$ is the product of either $\mathcal{C}_i$'s, in case of $i$ being a starting vertex of a movable part in $p$, or images of elements of $\mathcal{C}_i$ for the other vertices. Thus the dimension of $\mathcal{C}_p$ is the sum of dimensions of $\mathcal{C}_i$'s for all $i$ that start a movable part in $p$.
	\end{proof}
	Next, we shall focus on the closures of the cells. One of the more crucial results regarding the closed cells is that they are smooth, which is exactly what we shall prove now. This is a miracle of subrepresentation vector being $\mathbf{1}=(1,\dots,1)$. Even for $J$ being a single Jordan blocks if we choose other dimension vectors we get singular closed BB cells as was observed in \cite[Remark 3.15]{TotallyNonnegativeGrassmanniansGrassmannNecklacesAndQuiverGrassmannians}.
	\begin{tw}\label{ClBBCellsSmooth}
		Let $p\in X^{T}$ be a fixed point. Then the closure of the BB-cell $\mathcal{C}_p$ is a product of smooth quiver Grassmannians.
	\end{tw}
	\begin{proof}
		We will provide an explicit description of these subvarieties. First consider the case where $p\in X^T$ is a single movable part. Again, denote by $\clBB{i}\subseteq\proj{}{N-1}$ the closed BB-cell associated to the fixed point $p_i$ for $i\in\mathbb{Z}/n\mathbb{Z}$. Each $\mathcal{C}_i$ is isomorphic to a projective space $\proj{}{d_i-1}$ by forgetting all zero's imposed on vectors spanning lines in $\mathcal{C}_i$, so
		\begin{equation*}
			d_i = |\{v_k^s:v_k^s\ge p_i\}|.
		\end{equation*}
		Let $R_p\in\text{rep}A_n$ be the representation of the equioriented $A_n$-quiver defined by
		\begin{equation*}
			R_p: \mathbb{C}^{d_0}\xrightarrow{\pi_0}\mathbb{C}^{d_1}\xrightarrow{\pi_1}\mathbb{C}^{d_2}\to\dots\to\mathbb{C}^{d_n},
		\end{equation*}
		where $\pi_i:\mathbb{C}^{d_i}\to\mathbb{C}^{d_{i+1}}$ is the map induced from $J$ after forgetting these basis vectors $v_j^s$ such that $v_j^s<p_i$. In particular the $\pi_i$'s are epimorphisms! By forgetting all zeros in the description of $\clBB{p}$ we get an isomorphism
		\begin{equation*}
			\clBB{p}\cong \qgr{1}{R_p}.
		\end{equation*}
		Since all arrows in $R_p$ are epimorphisms, then the decomposition of $R_p$ into irreducible components is a decomposition into components that are all supported at the first vertex. As these are all injective, we get that $\text{Ext}^1(R_p,R_p)=0$. From the known theory of acyclic quiver Grassmannians (\cite[Corollary 4]{OnTheQuiverGrassmannianInTheAcyclicCase}) we get that $\qgr{1}{R_p}$ is smooth, irreducible and of dimension
		\begin{equation*}
			\dim \qgr{1}{R_p} = \langle \mathbf{1},\mathbf{d}-\mathbf{1}\rangle_{A_{m+1}} = \sum_{i=0}^m(d_i-1) - \sum_{i=0}^{m-1}(d_i-1) = d_m-1.
		\end{equation*}
		If $p$ is comprised of many movable parts, then since each movable part ends at a vector space $p_i$ such that $Jp_i=0$, we have that $\clBB{p}$ is the product of quiver Grassmannians associated to each movable part alone.
	\end{proof}
	Another crucial property of this BB-decomposition is that the closed cells form a stratification. Since the quiver Grassmannian $X$ has a stratification by affine cells, then it is equivariantly formal, which is one of the ingredients baked into the definition of GKM spaces. It has first been proven in \cite[Theorem 5.7]{GKMTheoryForCyclicQuiverGrassmannians}.
	\begin{stw}
		The BB-decomposition of $X$ forms a stratification, i.e. if for $p,q\in X^T$ we have $\mathcal{C}_p\cap \clBB{q}\neq\emptyset$, then $\mathcal{C}_p\subseteq\clBB{q}$.
	\end{stw}
	\begin{proof}
		This follows immediately from the fact that if we fix $i\in\mathbb{Z}/n\mathbb{Z}$ and pick two fixed points $p,q\in X^T$, then both $p_i$ and $q_i$ are fixed points in $\proj{}{N-1}$ under the $\mathbb{C}^\times$-action defined over the $i$-th vertex of $\Delta_n$. In particular if $\mathcal{C},\mathcal{D}\subseteq\proj{}{N-1}$ are the BB-cells corresponding to $p_i$ and $q_i$ respectively in the ambient space, then
		\begin{align*}
			& \mathcal{C}\cap X = (\mathcal{C}_p)_i,\\
			& \mathcal{D}\cap X = (\mathcal{C}_q)_i.
		\end{align*}
		These larger BB-cell do form a stratification, so if $\mathcal{C}_p\cap\clBB{q}\neq\emptyset$, then $\mathcal{C}_p\subseteq\clBB{q}$, since we are actually comparing the larger BB-cells.
	\end{proof}
	The partial order on fixed points defined in Definition $\ref{PartialOrder}$ lets us neatly describe the partial order of the closed BB-cells. 
	\begin{tw}\label{OrderOfBBCells}
		For two fixed points $p,q\in X^T$ we have $\clBB{p}\subseteq\clBB{q}$ if and only if $q\le p$.
	\end{tw}
	\begin{proof}	
		This is true for the BB-cells of all $\proj{}{N-1}$ as the lexicographic order is the order of weights acting on this space.
	\end{proof}
	When it comes to calculations in cohomology this becomes combinatorially intense, even if we restrict ourselves to just a single Jordan block. This particular case has been studied in \cite[Appendix B]{TotallyNonnegativeGrassmanniansGrassmannNecklacesAndQuiverGrassmannians}. We shall give a few observations about this case below.
	\begin{przyk}\label{SingleBlockRepresentation}
		Suppose that $J$ is a single Jordan block of size $N$. In this case the BB-decomposition simplifies heavily as we do not have to worry about other Jordan blocks. Our one-dimensional torus acts on each $\mathbb{C}^N$ in $\mathcal{M}$ in the standard way.\\
		
		A fixed point $p\in X^T$ in this case corresponds uniquely to a non-empty subset $I_p\subseteq\mathbb{Z}/n\mathbb{Z}$ -- the subset of these vertices, where a movable part in $p$ ends. In other words, we have
		\begin{equation*}
			I_p = \{i\in\mathbb{Z}/n\mathbb{Z}:Jp_i=0\}.
		\end{equation*}
		Subsets that define a fixed point must be non-empty and the difference of two consecutive points in such a subset must be at most $N$. The partial order from Definition $\ref{PartialOrder}$ simplifies to the usual order on the power set of $\mathbb{Z}/n\mathbb{Z}$.\\
		
		The open BB-cells are, of course, affine spaces; their dimension obtained from Proposition $\ref{BBCellsAffine}$ simplifies to
		\begin{equation*}
			\dim \mathcal{C}_p = |I_p^c|.
		\end{equation*}
		Movable parts in fixed points are parametrized by fixing a vertex (for example fixing the ending vertex of a segment) and length. 
		
		Observe that in this case we can restrict ourselves only to the case where $N\le n$.
		\begin{obs}
			Let $\mathcal{M}_N$ be the representation of $\Delta_n$, which has the same matrix over each arrow $J$ that is a single Jordan block of size $N$. If $N\ge n$, then we have
			\begin{equation*}
				\qgr{1}{\mathcal{M}_N}\cong\qgr{1}{\mathcal{M}_n}.
			\end{equation*}
			If $N\le N'$, then there exists an embedding
			\begin{equation*}
				\qgr{1}{\mathcal{M}_N}\hookrightarrow\qgr{1}{\mathcal{M}_{N'}}.
			\end{equation*}
		\end{obs}
		\begin{proof}
			To see the first part let us assume that $M\in\qgr{1}{\mathcal{M}_N}$ with $N\ge n$ and fix a vertex $i\in\mathbb{Z}/n\mathbb{Z}$. Let $M_i=\langle u\rangle$ and define 
			\begin{equation*}
				j = \min(k:u_k\neq 0).
			\end{equation*}
			After applying $J$ to $u$ $n$ times we arrive at the condition $J^nu\in\langle u \rangle$, note that either $j+n>N$, in which case $J^nu=0$ or $j+n\le N$ and then
			\begin{equation*}
				\min(k:(J^nu)_k\neq0) = j+n.
			\end{equation*}
			For the condition $J^nu\in\langle u\rangle$ to be met we must thus have $J^nu=0$ and this is equivalent to $u$ having $0$ as its first $N-n$ coordinates. Thus any element of $\qgr{1}{\mathcal{M}_N}$ is in fact a subrepresentation of $\mathcal{M}_n$.
			
			For $N\le N'$ the embedding is simply given by the map $\mathcal{M}_N\to\mathcal{M}_{N'}$, which is induced by the map $\mathbb{C}^N\to\mathbb{C}^{N'}$ that embeds $\mathbb{C}^N$ into the higher dimensional space by putting an appropriate amount of zeroes at the start.
		\end{proof}
		We can endow $X^{T}$ with a (commutative) semigroup structure by letting $p\cap q$ be the fixed point defined by the union of subsets
		\begin{equation*}
			I_{p\cap q} := I_p \cup I_q\subseteq\mathbb{Z}/n\mathbb{Z},
		\end{equation*}
		where $I_p,I_q\subseteq\mathbb{Z}/n\mathbb{Z}$ are the subsets corresponding to $p$ and $q$ respectively. Observe that for all $i\in\mathbb{Z}/n\mathbb{Z}$ we have
		\begin{equation}\label{CapOfPoints}
			(p\cap q)_i=\min(p_i,q_i),
		\end{equation}
		where $\min(p_i,q_i)$ is the point in $\proj{}{N-1}$ spanned by the vector that is lesser of vectors spanning $p_i$ and $q_i$.
		
		This construction lets us describe the intersections of the closed BB-cells
		\begin{stw}
			Let $p,q\in X^{T}$ be two fixed points. Then the following formula holds
			\begin{equation*}
				\clBB{p}\cap\clBB{q} = \clBB{p\cap q}.
			\end{equation*}
		\end{stw}
		\begin{proof}
			This is immediate when we consider BB-cells on $\proj{}{N-1}$ for fixed $i\in\mathbb{Z}$, since we have $\eqref{CapOfPoints}$.
		\end{proof}
	\end{przyk}
	\section{GKM structure}\label{GKMStructure}
	As we have already recalled in Section 3, the $T$-variety $X$ is a GKM variety. In the traditionally defined GKM graph the set of fixed points $X^T$ is the set of vertices and the set of one-dimensional orbits is the set of edges. This graph is not oriented, but we can orient it in our case using the $\mathbb{C}^\times$-action on $X$, as in \cite{GKMTheoryForCyclicQuiverGrassmannians}. We give a precise definiton below.
	\begin{defi}\label{Orientation}
		Let $X$ be a GKM variety with respect to a fixed $T$-action and ${\chi:\mathbb{C}^\times\to T}$ be a cocharacter such that the induced $\mathbb{C}^\times$-action on $X$ via $\chi$ has the same fixed point set
		\begin{equation*}
			X^T = X^{\mathbb{C}^\times}.
		\end{equation*}
		Let $x,y\in X^T$ be two fixed points such that they are connected by an edge, i.e. we have a one-dimensional orbit $O$ such that $x,y\in\overline{O}$. We orient this edge as $x\to y$ if for any (and thus all) $\nu\in O$ we have
		\begin{equation*}
			\lim_{\lambda\to0}\lambda\cdot \nu = x,
		\end{equation*}
		where $\lambda\in\mathbb{C}^\times$ acts on $\nu$ via $\chi$.
	\end{defi} 
	In particular if an arrow $p\to q$ exists, then for the orbit $O$ that induces this arrow we have $O\subseteq\mathcal{C}_p$. In \cite{GKMTheoryForCyclicQuiverGrassmannians} there is a description of one-dimensional orbits of any quiver Grassmannian endowed with their torus action and in our case this description becomes simpler.
	
	We can encode movable parts using Young tableaux. If $\mathcal{J}$ is the Young tableau of shape corresponding to the partition corresponding to $J$, then a movable part $S$ of length $m$ and from the $l$-th block is a coloring of $\mathcal{J}$, where we color in first $m+1$ boxes in the $l$-th row in $\mathcal{J}$ and leave the rest blank.
	\begin{przyk}\label{YoungTableauExample}
		Consider $J=J_4\oplus J_3\oplus J_3\oplus J_2\oplus J_1$ and a movable part of length $1$ in the second block. Its coloring is
		\ytableausetup{baseline,boxsize=1em}
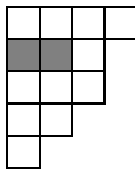
		\begin{equation*}
			\begin{ytableau}
				\, & \, &  & \, \\
				*(gray) & *(gray) & \, \\
				\, & \, & \, \\
				\, & \, \\
				\,
			\end{ytableau}
		\end{equation*}
	\end{przyk}
	These colorings let us visualize fundamental mutations (defined in \cite[Definition 6.9]{GKMTheoryForCyclicQuiverGrassmannians}), i.e. edges in the GKM graph of $X$.
	\begin{defi}\label{MovablePartMutation}
		Let $S$ be a movable part of length $m$ in some fixed point and $\lambda$ be the coloring of $\mathcal{J}$ associated with $S$. A mutation of $S$ is either
		\begin{enumerate}
			\item A single coloring $\lambda'$, whose colored segment is the same length as in $\lambda$, but in a lower row, or
			\item Two colorings $\lambda_1,\lambda_2$, where $\lambda_1$ is a coloring of the same row and length $m'$ and $\lambda_2$ is a coloring of any row and first $m''$ elements so that
			\begin{equation*}
				m'+m''=m+1.
			\end{equation*}
			In other words, we split the colored boxes in $\lambda$ into two colorings that correspond to movable parts, with the restriction that at least one movable part must be from the same block as $\lambda$.
		\end{enumerate}
	\end{defi}
	\begin{przyk}
		Consider the movable part from Example $\ref{YoungTableauExample}$. Then there are two possible mutations of Type 1. Namely, we have
		\ytableausetup{baseline, boxsize=1em}
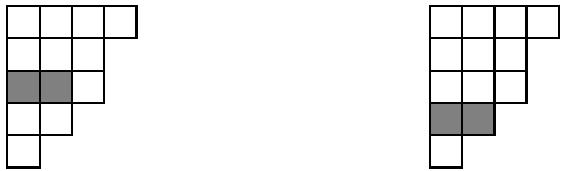
		\begin{align*}
			\begin{ytableau}
				\, & \, &  & \, \\
				\, & \,& \, \\
				*(gray) & *(gray) & \, \\
				\, & \, \\
				\,
			\end{ytableau} && \begin{ytableau}
				\, & \, &  & \, \\
				\, & \, & \, \\
				\, & \, & \, \\
				*(gray) & *(gray) \\
				\,
			\end{ytableau}
		\end{align*}
		There are also five possible mutations of Type 2. 
		\ytableausetup{baseline, boxsize=1em}
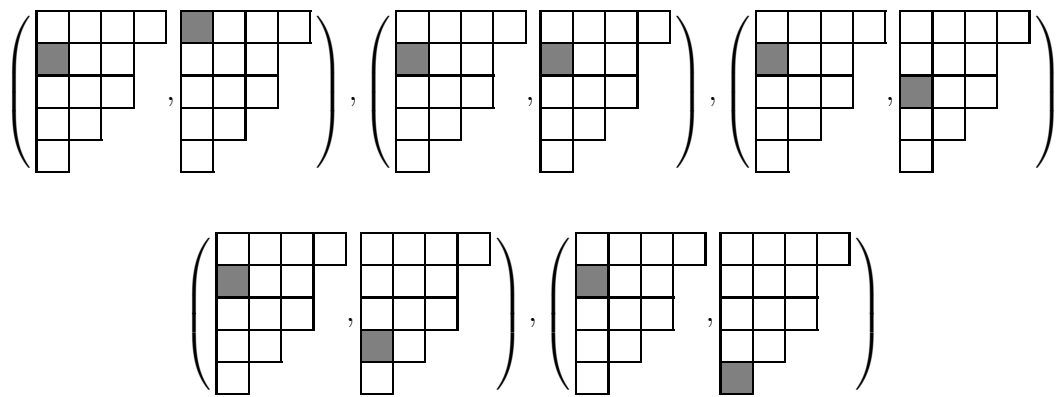
		\begin{gather*}
			\left( \, \begin{ytableau}
				\, & \, &  & \, \\
				*(gray) & \, & \, \\
				\, & \, & \, \\
				\, & \, \\
				\,
			\end{ytableau},\begin{ytableau}
				*(gray) & \, &  & \, \\
				\, & \,& \, \\
				\, & \, & \, \\
				\, & \, \\
				\,
			\end{ytableau} \, \right) , \left( \, \begin{ytableau}
				\, & \, &  & \, \\
				*(gray) & \, & \, \\
				\, & \, & \, \\
				\, & \, \\
				\,
			\end{ytableau},\begin{ytableau}
				\, & \, &  & \, \\
				*(gray) & \,& \, \\
				\, & \, & \, \\
				\, & \, \\
				\,
			\end{ytableau} \, \right) , \left( \, \begin{ytableau}
				\, & \, &  & \, \\
				*(gray) & \, & \, \\
				\, & \, & \, \\
				\, & \, \\
				\,
			\end{ytableau},\begin{ytableau}
				\, & \, &  & \, \\
				\, & \,& \, \\
				*(gray) & \, & \, \\
				\, & \, \\
				\,
			\end{ytableau} \, \right) \\
			\, \\
			\left( \, \begin{ytableau}
				\, & \, &  & \, \\
				*(gray) & \, & \, \\
				\, & \, & \, \\
				\, & \, \\
				\,
			\end{ytableau},\begin{ytableau}
				\, & \, &  & \, \\
				\, & \,& \, \\
				\, & \, & \, \\
				*(gray) & \, \\
				\,
			\end{ytableau} \, \right) , \left( \, \begin{ytableau}
				\, & \, &  & \, \\
				*(gray) & \, & \, \\
				\, & \, & \, \\
				\, & \, \\
				\,
			\end{ytableau},\begin{ytableau}
				\, & \, &  & \, \\
				\, & \,& \, \\
				\, & \, & \, \\
				\, & \, \\
				*(gray)
			\end{ytableau} \, \right)
		\end{gather*}
	\end{przyk}
	One can imagine the colored boxes as, for example, a stack of cards in a game of solitaire. Mutations correspond to either moving the whole stack "downwards" or making two stacks, but we can put the new stack wherever we want (we can put the stacks on top of each other in this scenario). These actions characterize exactly the edges of the GKM graph of $X$ as it is illustrated by the following proposition.
	\begin{stw}\label{GKM}
		Let $p,q\in X^T$ be two fixed points. Then an arrow $p\to q$ in the GKM graph of $X$ exists if and only if $q$ is obtained from $p$ by a mutation of exactly one movable part in $p$.
	\end{stw}
	Before the proof we shall give an example of a GKM graph.
	\begin{przyk}\label{GKMGraphn=3}
		Let $J$ be a single Jordan block of size $3$ and suppose we have $n=3$ vertices. We will denote the fixed points by the subsets of $\{0,1,2\}$ that they define. Mutations in this case can only correspond to the action of dividing movable parts into two, which on the level of subsets correspond to adding a single element to the subset defining a fixed point. The GKM graph of $X$ in this case turns out to be
		\begin{equation*}
			\resizebox{0.9\textwidth}{!}
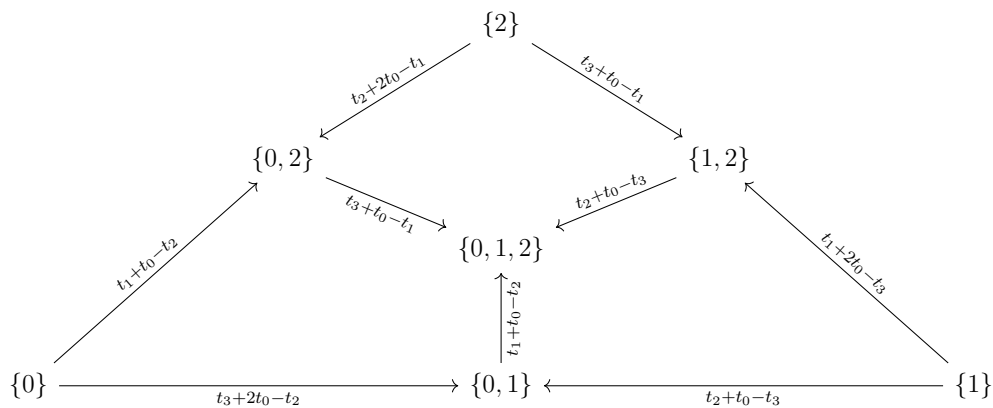
{
				\begin{tikzcd}[ampersand replacement=\&]
					\&\&\&\&\& {\{2\}} \\
					\\
					\&\&\& {\{0,2\}} \&\&\&\& {\{1,2\}} \\
					\&\&\&\&\& {\{0,1,2\}} \\
					\\
					{\{0\}} \&\&\&\&\& {\{0,1\}} \&\&\&\&\& {\{1\}}
					\arrow["{t_2+2t_0-t_1}", from=1-6, to=3-4,sloped]
					\arrow["{t_1+t_0-t_2}", from=6-1, to=3-4,sloped]
					\arrow["{t_3+2t_0-t_2}"', from=6-1, to=6-6,sloped]
					\arrow["{t_3+t_0-t_1}", from=1-6, to=3-8,sloped]
					\arrow["{t_1+2t_0-t_3}" , from=6-11, to=3-8,sloped]
					\arrow["{t_2+t_0-t_3}"', from=6-11, to=6-6,sloped]
					\arrow["{t_1+t_0-t_2}"', from=6-6, to=4-6,sloped]
					\arrow["{t_3+t_0-t_1}"', from=3-4, to=4-6,sloped]
					\arrow["{t_2+t_0-t_3}", from=3-8, to=4-6,sloped]
				\end{tikzcd}
			}
		\end{equation*}
	\end{przyk}
	\begin{proof}[Proof of Proposition $\ref{GKM}$]
		For two fixed points $p,q\in X^T$ let $p+q$ be the subrepresentation given by
		\begin{equation*}
			(p+q)_i = \langle v_{i}^l+v_j^s \rangle,
		\end{equation*}
		where $i\in\mathbb{Z}/n\mathbb{Z}, p_i=\langle v_i^l\rangle$ and $q_i=\langle v_j^s\rangle$.
		
		It is easy to see that if we assume the conditions on $I_p$ and $I_q$, then $p_i=q_i$ for all vertices besides those that correspond to the movable part in $p$ that gets mutated to create $q$. On this movable part we have $p_i<q_i$. To see the arrow $p\to q$, we just take the orbit $T\cdot(p+q)$. The representation $p+q$ has one non-zero coordinate everywhere besides one movable part in $p$, so the orbit of this element is one-dimensional and 
		\begin{align*}
			&\lambda\cdot (p+q) \xrightarrow[]{\lambda\to 0}p, 
			&\lambda\cdot(p+q) \xrightarrow[]{\lambda\to\infty}q.
		\end{align*}
		Now suppose that there exists an orbit $O\subseteq X$ such that $p,q\in\overline{O}$ and $O\subseteq\mathcal{C}_p$. As for any $\nu\in O$ we have $\lambda\cdot \nu\xrightarrow[]{\lambda\to\infty}q$, then we must have $\clBB{q}\subseteq\clBB{p}$ and by Theorem $\ref{OrderOfBBCells}$ we have $p_i\le q_i$ for all $i$. Let us pick an element $\nu\in O\subseteq \mathcal{C}_p$. By the explicit description of the BB-cells in Proposition $\ref{BBCellsAffine}$ we know that $\nu$ is determined by vector spaces put over the vertices which are starting movable parts in $p$. As the orbit $O=T\cdot \nu$ has to be one-dimensional, then we can only have one movable part in $p$ over which $q$ differs from $p$, since we would have an orbit of larger dimension otherwise. We can therefore assume that $p$ is a single movable part, say starting at the $i$-th vertex. The whole $T$-orbit is purely determined by what happens at the $i$-th vertex, since lines in other vertices are just images by $J$. This reduces the case to $\proj{}{N-1}$. We know the $1$-dimensional orbits there -- in the GKM graph of $\proj{}{N-1}$ under the induced torus action an edge $p_i\to q_i$ exists if and only if $p_i< q_i$. Let $p_i=\langle v_j^l\rangle, q_i=\langle v_k^s\rangle$ and $\lambda_p,\lambda_q$ be colorings of $\mathcal{J}$ associated to $p$ and $q$ respectively. We can distinguish two cases
		\begin{enumerate}
			\item If $j=k$, then we must have $l<s$. In this case $\lambda_q$ came from $\lambda_p$ by moving the whole colored row in $\lambda_p$ down.
			\item If $j<k$, then $\lambda_q$ came about from dividing the colored row in $\lambda_p$ into two parts. We split this row at the $(i+j_s-j)$-th vertex.
		\end{enumerate}
	\end{proof}
	We also need to describe the labels of the edges of the GKM graph, which is done below.
	\begin{stw}\label{GKMedges}
		Suppose that in the GKM graph of $X$ there exists an edge $p\to q$ between two fixed points $p,q\in X^T$. Then the label of this edge can be chosen to be the following polynomial
		\begin{equation*}
			w(p,q) = t_{m+j,l} + (j_l-j_k+i-m)t_0 - t_{i+j,l},
		\end{equation*}
		where $q$ comes by mutating a movable part in $p$ that starts at the $j$-th vertex, is of length $m$ and comes from the $l$-th block. The augmentation occurs at the $(i+j)$-th vertex.
	\end{stw}
	\begin{proof}
		We need the weight of the $T$-action on which we act on the $\proj{}{1}$ that is equiavariantly isomorphic to the closed orbit $\overline{T\cdot(p+q)}$. The isomorphism $\overline{T\cdot(p+q)}\to\proj{}{1}$ is given by looking at the vector space lying over the starting point of the movable part in $p$ which we augment to obtain $q$. 
		
		Suppose first that this movable part in $p$ starts at the $0$-th vertex, is of length $m$ and comes from the $l$-th block; in particular $p_i=\langle v_{j_l-m}^l\rangle$. Suppose further that $q_i=\langle v_{j_k-i}^k$, i.e. the augmentation occurs at the $i$-th vertex and at one of the movable parts is from the $k$-th block. The torus acts on these vectors with weights
		\begin{align*}
			& w_1 = t_0^{j_l-m-1}t_{m,l},\\
			& w_2 = t_0^{j_k-i-1}t_{i,l}.
		\end{align*}
		So the torus acts on the closed orbit identified with $\proj{}{1}$ with weight
		\begin{equation*}
			w = t_{m,l} + (j_l-j_k+i-m)t_0 - t_{i,l}.
		\end{equation*}
		The above polynomial is also a weight in the torus representation $T_q\clBB{p}$. Now, if the movable part in $p$ that we are augmenting starts at the $j$-th vertex, then we have to translate everything by $j$:
		\begin{equation*}
			w(p,q) = t_{m+j,l} + (j_l-j_k+i-m)t_0 - t_{i+j,l}.
		\end{equation*}
		This is the label of the edge $p\to q$ in the GKM graph of $X$ and a weight in the torus representation $T_q\clBB{p}$.
	\end{proof}
	\begin{uw}\label{CyclicAction}
		There is a natural $(\mathbb{Z}/n\mathbb{Z})$-action on $X$, which comes from rotating the quiver along the direction given by the arrows in $\Delta_n$. Since the $\mathbb{C}^\times$-action defined by Formula $\eqref{C-action}$ is independent of the vertex over which we define this action, then the $\mathbb{Z}/n\mathbb{Z}$ is immediately seen to lift to the set of fixed point $X^T$ as well as the set of (closed) BB-cells. The $T$-action on $X$ commutes with the $(\mathbb{Z}/n\mathbb{Z})$-action up to a twist by $\rho:T\to T$ and so, similarly as is the case with Schubert classes in the usual Grassmannians, this gives a particular symmetry on the GKM-graph of $X$. Namely, if $\alpha\in H_T^*(X)$ is a class, then for the generator $\tau\in\mathbb{Z}/n\mathbb{Z}$ we can set
		\begin{equation}\label{Z/nZ-ActionOnGraph}
			(\tau\cdot\alpha)|_p = \alpha|_{\tau(p)}\circ \rho,
		\end{equation}
		where $\rho:H_T^*(*)\to H_T^*(*)$ is the map induced by $\rho:T\to T$ defined by $\eqref{Rotation}$.
	\end{uw}
	With this section concluded we are ready to delve into the analysis of the multiplcative structure of $H_T^*(X)$.
	\section{Multiplicative structure of equivariant cohomology}\label{MultiplicativeStructure}
	This last section is devoted to analyzing the multiplicative structure of the equivariant cohomology ring $H_T^*(X)$. M. Lanini and A. P{\"u}tz provide a construction of the so called Knutson-Tao basis for equivaraint cohomology \cite[Theorem 3.9]{PermutationActionsOnQuiverGrassmannians} for certain quiver Grassmannians. This basis is dependent on a chosen orientation of the GKM-graph of $X$ using cocharacters as we recalled in Definition $\ref{Orientation}$. We give a precise definition below
	\begin{defi}[{\cite[Definition 2.12]{PermutationRepresentationsOnSchubertVarieties}}]\label{Knutson-TaoBasis}
		Let $X$ be any GKM variety under a fixed $T$-action and $\chi:\mathbb{C}^\times\to T$ be a cocharacter such that the induced $\mathbb{C}^\times$-action by $\chi$ preserves fixed points $X^T=X^{\mathbb{C}^\times}$. This induced $\mathbb{C}^\times$-action determines an orientation on the GKM graph of $X$ as in Definition $\ref{Orientation}$. A set $\{q^x:x\in X^T\}$ is called a \textbf{Knutson-Tao basis} if its elements satisfy the following
		\begin{enumerate}
			\item For each $y\in X^T$ the polynomial $q^x|_y$ is either $0$ in the case that there is no oriented path from $y$ to $x$ or is homogeneous of degree $\deg(q^x|_x)$ in the case where there is such a path.
			\item The restriction $q^x|_x$ is the product of weights associated to edges going from $x$ in the GKM-graph of $X$.
		\end{enumerate}
	\end{defi}
	If a Knutson-Tao basis exists, then it is a basis of $H_T^*(X)$ as an $H_T^*(*)$-module \cite[Proposition 2.13]{PermutationRepresentationsOnSchubertVarieties}. If our GKM variety $X$ is Palais-Smale, meaning that $X$ satisfies hypothesis of Definition $\ref{Knutson-TaoBasis}$ and for each edge $x\to y$ in the GKM graph of $X$ the amount of edges going out of $x$ is strictly higher than the amount of edges going out of $y$, then this basis is unique \cite[Lemma 2.16]{PermutationRepresentationsOnSchubertVarieties}. 
	
	This basis behaves well with respect to the order on $X^T$ given by paths in the graph. In our case this order coincides with the inverted lexicographic order. These basis' arose from the Schubert basis for equivariant cohomology of the usual Grassmannians in \cite{PuzzlesAndEquivariantCohomologyOfGrassmannians}. Similarly as for flag varieties and usual Grassmannians, a central question regarding these basis is how they multiply. Since $\{q^x:x\in X^T\}$ is a basis of $H_T^*(X)$ as a $H_T^*(*)$-module, then computation of the structure constants
	\begin{equation*}
		q^x \cdot q^y = \sum_{z\in X^T}c_{x,y}^zq^z
	\end{equation*}
	determines the multiplicative structure of $H_T^*(X)$, considered as a $H_T^*(*)$-algebra. 
	
	In our case each cell $\clBB{x}$ gives an equivariant fundamental class $[\clBB{x}]\in H_{*}^T(X)$ and the set of fundamental classes is a basis of equivariant homology. By Theorem $\ref{ClBBCellsSmooth}$ we know the closed BB cells are smooth and so integration along such a cell is much simpler to grasp. Using this pairing we can get the dual basis $\{p^x\in H_T^*(X):x\in X^T\}$, i.e. characterized by
	\begin{equation*}
		\int_{\clBB{y}}p^x|_{\clBB{y}} = \delta_{x,y}.
	\end{equation*}
	Using Atiyah-Bott, Berline-Vergne we can further rewrite the above formula as
	\begin{equation*}
		\int_{\clBB{y}}p^x|_{\clBB{y}} = \sum_{z\in X^T:z\in\clBB{y}}\frac{p^x|_z}{e(T_z\clBB{y})} = \delta_{x,y},
	\end{equation*}
	where $e(T_z\clBB{y})$ is the product of weights of the torus representation $T_z\clBB{y}$. These products of weights can be read directly from the GKM graph of $X$. Namely, for a fixed point $x\in X^T$ we first consider the subgraph of $X$ spanned by all oriented paths starting from $x$. For any $y\in X^T$ such that $y\in\clBB{x}$ we have that the vertex corresponding to $y$ lies in this subgraph. To get the Euler class of the torus representation $T_y\clBB{x}$ we take the product of labels of edges either from or into $y$ in this subgraph. However, one has to be careful with the sign of these labels -- normally, labels of edges in the GKM graph are defined up to a scalar multiple and since we need weights of the torus representation $T_y\clBB{x}$, then one has to take the appropriately scaled generator of the ideal that generates labels of edges in the GKM graph of $X$; in Proposition $\ref{GKMedges}$ we give the appropriately scaled labels.
	\begin{przyk}[{\cite[Example 5.6]{TotallyNonnegativeGrassmanniansGrassmannNecklacesAndQuiverGrassmannians}}]
		Take $n=3$ and $J=J_3$. Then there are $7$ fixed points associated to non-empty subsets of $\{0,1,2\}$. The GKM graph of $X$ in this case is given in Example $\ref{GKMGraphn=3}$. The Knutson-Tau basis of $H_T^*(X)$ consists of three types of elements:
		\begin{enumerate}
			\item The identity $1\in H_T^*(X)$, dual to the smallest cell $\clBB{\{0,1,2\}}$.
			\item Three elements dual to the medium-sized cells, i.e. those corresponding to two-element subsets
			\begin{align*}
				& p^{\{0,1\}}=(t_1+t_0-t_2,t_1+2t_0-t_3,0,t_1+t_0-t_2,0,0,0),\\
				& p^{\{1,2\}}=(0,t_2+t_0-t_3,t_2+2t_0-t_1,0,t_2+t_2-t_3,0,0),\\
				& p^{\{0,2\}}=(t_3+2t_0-t_2,0,t_3+t_0-t_1,0,0,t_3+t_0-t_1,0),
			\end{align*}
			where the order follows the following order on fixed points
			\begin{equation*}
				(\{0\},\{1\},\{2\},\{0,1\},\{1,2\},\{0,2\},\{1,2,3\}).
			\end{equation*}
			\item Three elements dual to the largest cells:
			\begin{align*}
				& p^{\{0\}}=((t_2-2t_0-t_3)(t_2-t_0-t_1),0,\dots,0),\\
				& p^{\{1\}}=(0,(t_3-2t_0-t_1)(t_3-t_0-t_2),0,\dots,0),\\
				& p^{\{2\}}=(0,0,(t_1-2t_0-t_2)(t_1-t_0-t_3),0,\dots,0).
			\end{align*}
		\end{enumerate}
	\end{przyk}
	From simple linear algebra one can immediately observe that this dual basis always is a Knutson-Tao basis. Authors in \cite[Theorem 3.22]{PermutationActionsOnQuiverGrassmannians} define a family of quiver Grassmannians for which a Knutson-Tao basis exists (and is unique); our GKM space $X$ is a member of this family.
	
	Elements of this dual basis are homogeneous (we grade $H_T^*(X)$ in the usual fashion, i.e. $\deg t_0=\deg t_{i,j}=2$ for all $i,j$); degree of $p^x$ is given by the dimension of the cell it corresponds to
	\begin{equation*}
		\deg(p^x) = 2\dim(\clBB{x}).
	\end{equation*}
	By definition of the dual basis $p^x$, the structure constants are given by integrals
	\begin{equation*}
		c_{x,y}^z = \int_{\clBB{z}}p^x|_{\clBB{z}}\cdot p^y|_{\clBB{z}} = \sum_{s\in X^T: \clBB{x}\cup\clBB{y}\subseteq\clBB{s}\subseteq \clBB{z}}\frac{p^x|_s\cdot p^y|_s}{e(T_s\clBB{z})},
	\end{equation*}
	We can outline that the basis $\{p^x:x\in X^T\}$ satisfies the same symmetry as the closed BB-cells that these classes are dual to, i.e. if $\tau\in\mathbb{Z}/n\mathbb{Z}$ is the generator, then
	\begin{equation*}
		p^{\tau(x)} = \tau(p^x),
	\end{equation*}
	where $\tau$ acts on fixed point by rotation of the quiver and on $H_T^*(X)$ by Formula $\eqref{Z/nZ-ActionOnGraph}$.
    \printbibliography 

@book{ElementsOfTheRepresentationTheoryOfAssociativeAlgebras,
	author = {Ibrahim Assem and Daniel Simson and Andrzej Skowroński},
	title = {Elements of the representation theory of associative algebras},
	publisher = {Cambridge University Press},
	year = {2007}
}

@article {OnTheQuiverGrassmannianInTheAcyclicCase,
	AUTHOR = {Caldero, Philippe and Reineke, Markus},
	TITLE = {On the quiver {G}rassmannian in the acyclic case},
	JOURNAL = {J. Pure Appl. Algebra},
	FJOURNAL = {Journal of Pure and Applied Algebra},
	VOLUME = {212},
	YEAR = {2008},
	NUMBER = {11},
	PAGES = {2369--2380},
	ISSN = {0022-4049,1873-1376},
	MRCLASS = {14M15 (14L30 16G10 16G20)},
	MRNUMBER = {2440252},
	MRREVIEWER = {Ivan\ Arzhantsev},
	DOI = {10.1016/j.jpaa.2008.03.025},
	URL = {https://doi.org/10.1016/j.jpaa.2008.03.025},
}

@article {GKM,
	AUTHOR = {Goresky, Mark and Kottwitz, Robert and MacPherson, Robert},
	TITLE = {Equivariant cohomology, {K}oszul duality, and the localization
	theorem},
	JOURNAL = {Invent. Math.},
	FJOURNAL = {Inventiones Mathematicae},
	VOLUME = {131},
	YEAR = {1998},
	NUMBER = {1},
	PAGES = {25--83},
	ISSN = {0020-9910,1432-1297},
	MRCLASS = {55N91 (14F25 14F32 16E99 18G10 55N33)},
	MRNUMBER = {1489894},
	MRREVIEWER = {Roy\ Joshua},
	DOI = {10.1007/s002220050197},
	URL = {https://doi.org/10.1007/s002220050197},
}

@incollection {3LecturesOnQuiverGrassmannians,
	AUTHOR = {Cerulli Irelli, Giovanni},
	TITLE = {Three lectures on quiver {G}rassmannians},
	BOOKTITLE = {Representation theory and beyond},
	SERIES = {Contemp. Math.},
	VOLUME = {758},
	PAGES = {57--89},
	PUBLISHER = {Amer. Math. Soc.},
	YEAR = {2020},
	ISBN = {978-1-4704-5131-8},
	MRCLASS = {16G20 (14D20)},
	MRNUMBER = {4186968},
	DOI = {10.1090/conm/758/15232},
	URL = {https://doi.org/10.1090/conm/758/15232},
}

@article {PermutationActionsOnQuiverGrassmannians,
	AUTHOR = {Lanini, Martina and P\"{u}tz, Alexander},
	TITLE = {Permutation actions on quiver {G}rassmannians for the
	equioriented cycle via {GKM}-theory},
	JOURNAL = {J. Algebraic Combin.},
	FJOURNAL = {Journal of Algebraic Combinatorics. An International Journal},
	VOLUME = {57},
	YEAR = {2023},
	NUMBER = {3},
	PAGES = {915--956},
	ISSN = {0925-9899,1572-9192},
	MRCLASS = {14M15 (05E14)},
	MRNUMBER = {4577426},
	MRREVIEWER = {Francesco\ Esposito},
	DOI = {10.1007/s10801-022-01211-5},
	URL = {https://doi.org/10.1007/s10801-022-01211-5},
}

@article {GKMTheoryForCyclicQuiverGrassmannians,
	AUTHOR = {Lanini, Martina and P\"{u}tz, Alexander},
	TITLE = {G{KM}-theory for torus actions on cyclic quiver
	{G}rassmannians},
	JOURNAL = {Algebra Number Theory},
	FJOURNAL = {Algebra \& Number Theory},
	VOLUME = {17},
	YEAR = {2023},
	NUMBER = {12},
	PAGES = {2055--2096},
	ISSN = {1937-0652,1944-7833},
	MRCLASS = {16G20 (14L30 14M15)},
	MRNUMBER = {4650390},
	DOI = {10.2140/ant.2023.17.2055},
	URL = {https://doi.org/10.2140/ant.2023.17.2055},
}

@article {DesingularizationsOfCyclicQuiverGrassmannians,
	AUTHOR = {P\"{u}tz, Alexander and Reineke, Markus},
	TITLE = {Desingularizations of quiver {G}rassmannians for the
	equioriented cycle quiver},
	JOURNAL = {Pacific J. Math.},
	FJOURNAL = {Pacific Journal of Mathematics},
	VOLUME = {326},
	YEAR = {2023},
	NUMBER = {1},
	PAGES = {109--133},
	ISSN = {0030-8730,1945-5844},
	MRCLASS = {14M15 (14E15)},
	MRNUMBER = {4678151},
	DOI = {10.2140/pjm.2023.326.109},
	URL = {https://doi.org/10.2140/pjm.2023.326.109},
}

@article{EveryProjectiveVarietyIsQuiverGrassmannian,
	author = {Markus Reineke},
	title = {Every projective variety is a quiver {G}rassmannian},
	journal = {Algebras and Representation Theory},
	volume = {16},
	number = {5},
	pages = {1313-1314},
	year = {2012}
}

@book{GeneralRepresentationsOfQuivers,
	author = {Aidan Schofield},
	title = {General representations of quivers},
	publisher = {Proceedings of the London Mathematical Society},
	year = {1992}
}

@article {QuiverGrassmanniansAssociatedWithStringModules,
	AUTHOR = {Cerulli Irelli, Giovanni},
	TITLE = {Quiver {G}rassmannians associated with string modules},
	JOURNAL = {J. Algebraic Combin.},
	FJOURNAL = {Journal of Algebraic Combinatorics. An International Journal},
	VOLUME = {33},
	YEAR = {2011},
	NUMBER = {2},
	PAGES = {259--276},
	ISSN = {0925-9899,1572-9192},
	MRCLASS = {16G20 (13F60)},
	MRNUMBER = {2765325},
	MRREVIEWER = {Kyungyong\ Lee},
	DOI = {10.1007/s10801-010-0244-6},
	URL = {https://doi.org/10.1007/s10801-010-0244-6},
}

@article {TotallyNonnegativeGrassmanniansGrassmannNecklacesAndQuiverGrassmannians,
	AUTHOR = {Feigin, Evgeny and Lanini, Martina and P\"{u}tz, Alexander},
	TITLE = {Totally nonnegative {G}rassmannians, {G}rassmann necklaces,
	and quiver {G}rassmannians},
	JOURNAL = {Canad. J. Math.},
	FJOURNAL = {Canadian Journal of Mathematics. Journal Canadien de
	Math\'{e}matiques},
	VOLUME = {75},
	YEAR = {2023},
	NUMBER = {4},
	PAGES = {1076--1109},
	ISSN = {0008-414X,1496-4279},
	MRCLASS = {16G20 (05E14)},
	MRNUMBER = {4620316},
	DOI = {10.4153/s0008414x22000232},
	URL = {https://doi.org/10.4153/s0008414x22000232},
}

@article {PermutationRepresentationsOnSchubertVarieties,
	AUTHOR = {Tymoczko, Julianna S.},
	TITLE = {Permutation representations on {S}chubert varieties},
	JOURNAL = {Amer. J. Math.},
	FJOURNAL = {American Journal of Mathematics},
	VOLUME = {130},
	YEAR = {2008},
	NUMBER = {5},
	PAGES = {1171--1194},
	ISSN = {0002-9327,1080-6377},
	MRCLASS = {14M15 (14F25 14F43)},
	MRNUMBER = {2450205},
	MRREVIEWER = {Harry\ Tamvakis},
	DOI = {10.1353/ajm.0.0018},
	URL = {https://doi.org/10.1353/ajm.0.0018},
}

@article{PuzzlesAndEquivariantCohomologyOfGrassmannians,
	AUTHOR = {Knutson, Allen and Tao, Terence},
	TITLE = {Puzzles and equivariant cohomology of {G}rassmannians},
	JOURNAL = {Duke Math. J.},
	FJOURNAL = {Duke Mathematical Journal},
	VOLUME = {119},
	YEAR = {2003},
	PAGES = {221-260},
	ISSN = {2},
}

@book{EquivariantCohomologyInAlgebraicGeometry,
	AUTHOR = {Anderson, David and Fulton, William},
	TITLE = {Equivariant cohomology in algebraic geometry},
	PUBLISHER = {Cambridge University Press},
	YEAR = {2023}
}
\end{document}